\documentclass[11pt]{amsart}

\usepackage{amsmath,amssymb,amsthm}
\usepackage[margin=1in]{geometry}
\usepackage{hyperref}
\usepackage{tikz}

\newtheorem{theorem}{Theorem}
\newtheorem{proposition}{Proposition}
\newtheorem{lemma}{Lemma}
\newtheorem{definition}{Definition}
\newtheorem{remark}{Remark}

\title[Decomposition-Closed Sublattices as Minimizer Sets]
{Decomposition-Closed Sublattices as Minimizer Sets\\ of Modular Functions over Distributive Lattices}

\author{\large{A}\small{HMET} \large{A}\small{LKAN}}
\address{Ahmet Alkan: Department of Economics, Sabanc{\i} University}
\email{alkan@sabanciuniv.edu.tr}
\author{\large{K}\small{EMAL} \large{Y}\small{ILDIZ}}
\address{Kemal Yildiz: Department of Economics, Bilkent University}
\email{kemal.yildiz@bilkent.edu.tr}
\date{\today}

\begin{document}

\begin{abstract}
We characterize the subsets of a finite distributive lattice that arise as the minimizer sets of modular functions. It is immediate that the minimizer set of any modular function is a decomposition-closed sublattice. Our main theorem establishes the converse: every decomposition-closed sublattice is the minimizer set of a modular function. Our proof is constructive. Using Birkhoff's representation, we decompose the problem along the intervals of an arbitrary maximal chain of the given sublattice. The key ingredient is a weight assignment on connected difference posets that yields local modular functions vanishing precisely at the interval endpoints. We also show by example that the characterization fails for nondistributive lattices.
\end{abstract}

\maketitle

\section{Introduction}

A natural problem in the theory of distributive lattices is to determine which sublattices arise as the minimizer sets of modular functions. Through Birkhoff's representation \cite{Birkhoff1937,DaveyPriestley}, every modular function admits a simple description in terms of weights on the join-irreducible elements.

The central notion introduced in this paper is that of a decomposition-closed sublattice. It is defined by a natural closure condition on decompositions of intervals whose endpoints belong to the sublattice. It is immediate that every minimizer set of a modular function is decomposition closed. Our main result establishes the converse: decomposition-closedness is exactly the property required.

We therefore obtain a complete characterization of the minimizer sets of modular functions on finite distributive lattices. Our proof is direct and constructive. Starting from an arbitrary maximal chain of the given sublattice, we build the desired modular function by combining explicit local constructions on the connected difference posets arising from the Birkhoff representation.

\section{Setup}

Let $(S,\le)$ be a finite lattice. A subset $Q$ of $S$ is the \emph{minimizer set} of $F : S \to \mathbb{R}$ if
\[
Q = \arg\min_{s \in S} F(s).
\]
A function $F : S \to \mathbb{R}$ is \emph{modular} if
\[
F(s) + F(s') = F(s \wedge s') + F(s \vee s') \qquad \text{for all } s, s' \in S.
\]

\begin{definition}\label{def:decomp}
Let $I = [a,b]$ denote the interval
\[
I = \{ x \in S : a \le x \le b \},
\]
where $a \le b$. A pair of elements $s, s' \in I$ is called a \emph{decomposition} of $I$ if
\[
s \wedge s' = a, \qquad s \vee s' = b.
\]
A sublattice $Q \subseteq S$ is \emph{decomposition closed} if, whenever $I = [a,b]$ is an interval of $S$ with $a, b \in Q$, every decomposition of $I$ is contained in $Q$.
\end{definition}

\begin{proposition}\label{prop:necessity}
If $Q$ is the minimizer set of a modular $F$ on $S$, then $Q$ is a decomposition-closed sublattice of $S$.
\end{proposition}

\begin{proof}
Normalize $F$ so $\min_{s \in S} F(s) = 0$; then $Q = F^{-1}(0)$. If $F(s) = F(s') = 0$, modularity gives
\[
F(s \wedge s') + F(s \vee s') = 0
\]
with both terms $\ge 0$, so both vanish: $F^{-1}(0)$ is a sublattice. If $a, b \in Q$, $a \le b$, and $s, s'$ form a decomposition of the interval $[a,b]$, then
\[
F(s) + F(s') = F(a) + F(b) = 0,
\]
so $F(s) = F(s') = 0$. Thus $F^{-1}(0)$ is decomposition closed. The converse fails for general lattices (Appendix~\ref{app:N5}), showing that the distributivity assumption in our main theorem is essential.
\end{proof}

\begin{theorem}\label{thm:main}
Every decomposition-closed sublattice $Q$ of a finite distributive lattice $S$ is the minimizer set of a modular function.
\end{theorem}

The characterization is constructive. Beyond establishing existence, the construction shows that every maximal chain of the given sublattice yields a modular function having the prescribed minimizer set. Appendix~\ref{app:example} works through the construction in full for a small example.

\section{Birkhoff Representation and Overview of Proof}

By the Birkhoff Representation Theorem \cite{Birkhoff1937,DaveyPriestley}, $(S,\le)$ is isomorphic, via $s \mapsto D_s = \{ r \in R : r \le s \}$, to the lattice of downsets of the poset $R$ of join-irreducible elements of $S$, with
\[
D_s \cap D_{s'} = D_{s \wedge s'}, \qquad D_s \cup D_{s'} = D_{s \vee s'}.
\]
We call $D_b \setminus D_a$, equipped with the induced order, the \emph{difference poset} of the interval $[a,b]$. A modular $F$ on $S$ is represented by a weight $g : R \to \mathbb{R}$ via
\[
F(s) = F(s^+) - \sum_{r \in D_{s^+} \setminus D_s} g(r)
\]
where $s^+ = \max S$, and conversely any such formula is modular, for any choice of $g$. Let $s^+$ denote $\max S$, $s^-$ denote $\min S$, $q^+$ and $q^-$ the max and min of $Q$, and $S_0 = [q^-, q^+] \subseteq S$ the interval between them.

An interval $[a,b]$ in $S$, with $a \le b$, is called a \emph{$Q$-interval} if $a, b \in Q$. A $Q$-interval $[a,b]$ is called \emph{$Q$-prime} if it contains no $q \in Q$ other than $a, b$.

\medskip
\noindent \textit{Overview of the Proof.}
We choose any maximal chain $q^+ = q_1 > \cdots > q_N = q^-$ in $Q$. Decomposition closedness implies that each $Q$-prime interval $[q_{n+1}, q_n]$ corresponds, under the Birkhoff representation, to a connected difference poset (Lemma~\ref{lem:connected}). On each such poset we construct a nonnegative modular function that vanishes only at the two endpoints of the interval (Lemma~\ref{lem:weight}). Their sum is a modular function that vanishes on all of $Q$ (Lemma~\ref{lem:vanish}) and is strictly positive on $[q^-, q^+] \setminus Q$ (Lemma~\ref{lem:positive}). Finally, adding two modular penalty functions outside this interval gives a modular function whose minimizer set is exactly $Q$.

\section{Lemma 1 (Connectedness of $Q$-Prime Intervals)}

\begin{lemma}\label{lem:connected}
If $Q$ is decomposition closed and $[a,b]$ is a $Q$-prime interval in $S$, then the difference poset $D_b \setminus D_a$ is connected, meaning that its undirected Hasse diagram is connected.
\end{lemma}

\begin{proof}
Let $D := D_b \setminus D_a$. First, $D$ is order-convex in $R$: if $x, y \in D$ and $x \le z \le y$, then $z \in D_b$ (as $D_b$ is downward closed and $z \le y \in D_b$), and $z \notin D_a$ (else $x \le z \in D_a$ would force $x \in D_a$, since $D_a$ is downward closed --- contradicting $x \in D$); so $z \in D$.

Suppose $D$ is disconnected: $D = D' \cup D''$, with $D', D''$ nonempty and no comparabilities between them. Since $D$ is order-convex, any two comparable elements of $D$ are joined by a chain lying entirely in $D$, hence in the same piece; so $D_a \cup D'$ and $D_a \cup D''$ are each downward closed in $R$ --- downward closure from a point of $D'$ (resp.\ $D''$) either stays in $D_a$ or lands back in $D'$ (resp.\ $D''$), never crossing to the other piece. By the Birkhoff representation, there exist $q', q'' \in S$ with
\[
D_{q'} = D_a \cup D', \qquad D_{q''} = D_a \cup D''.
\]
Since $D', D''$ are nonempty proper subsets of $D$, $a < q', q'' < b$, and
\[
D_{q'} \cup D_{q''} = D_a \cup D' \cup D'' = D_b, \qquad D_{q'} \cap D_{q''} = D_a,
\]
so $q' \vee q'' = b$ and $q' \wedge q'' = a$: $q'$ and $q''$ form a decomposition of the interval $[a,b]$. Since $[a,b]$ is $Q$-prime, decomposition closedness implies that $q'$ and $q''$ belong to $Q$, a contradiction. Hence $D$ is connected.
\end{proof}

\section{Lemma 2 (Connected Poset Weight Lemma)}

\begin{lemma}\label{lem:weight}
Let $P$ be a finite connected poset. Then there exists a function $g : P \to \mathbb{R}$ with $\sum_{p \in P} g(p) = 0$ such that $\sum_{p \in D} g(p) > 0$ for every nonempty proper downset $D \subseteq P$.
\end{lemma}

\begin{proof}
Let $T$ be a spanning tree of the Hasse diagram of $P$. Orient every edge upward (from the lower endpoint to the upper endpoint), and assign every edge weight $1$. For $v \in P$ set
\[
g(v) := \deg^+(v) - \deg^-(v),
\]
where $\deg^+(v)$ and $\deg^-(v)$ are the number of tree edges directed out of and into $v$, respectively. Each tree edge contributes $+1$ to the tail vertex and $-1$ to the head vertex in the sum $\sum_{v \in P} g(v)$. Hence every edge contributes zero overall, so
\[
\sum_{v \in P} g(v) = 0.
\]

Let $D$ be a nonempty proper downset. Because $T$ is connected and $D$ is a nonempty proper subset of its vertices, there exists at least one edge joining $D$ to $P \setminus D$. If such an edge $e = (p, p')$, $p \lessdot p'$, had its upper endpoint $p' \in D$, then $p \le p'$ would force $p \in D$ as well (as $D$ is downward closed), contradicting that only one endpoint of $e$ lies in $D$. So every such crossing edge has its lower endpoint in $D$ and its upper endpoint outside $D$, i.e., it leaves $D$.

Now sum $g$ over $D$: an edge with both endpoints in $D$ contributes $+1$ (to $\deg^+$ of its lower endpoint) and $-1$ (to $\deg^-$ of its upper endpoint), which cancel; an edge with neither endpoint in $D$ contributes nothing; and each crossing edge, leaving $D$, contributes only its $+1$ term. Hence
\[
\sum_{v \in D} g(v) = \#\{\text{tree edges leaving } D\},
\]
which is strictly positive, since $T$'s connectivity guarantees at least one crossing edge.

No choice of a root is required; the argument applies to any spanning tree of the Hasse diagram of an arbitrary connected poset.
\end{proof}

\begin{remark}\label{rem:converse}
The converse also holds: if such a weight assignment exists, then $P$ must be connected. This characterization is not needed for the proof of Theorem~\ref{thm:main}.
\end{remark}

The remainder of the proof first constructs a modular function on the interval $S_0 = [q^-, q^+]$, and then extends it to all of $S$.

\section{Local Construction on $S_0$}

Fix an arbitrary maximal chain $q_1 = q^+ > q_2 > \cdots > q_N = q^-$ in $Q$. Each consecutive pair on the chosen chain determines a $Q$-prime interval $I_n := [q_{n+1}, q_n]$, so by Lemma~\ref{lem:connected} the difference poset $P_n := D_{q_n} \setminus D_{q_{n+1}}$ is connected, and by Lemma~\ref{lem:weight} carries weights $g_n$ giving a modular function
\[
F_n(s) := \sum_{r \in D_s \cap P_n} g_n(r), \qquad s \in S.
\]
Since $D_s \cap P_n$ is a downset of $P_n$ for every $s \in S$, Lemma~\ref{lem:weight} implies that $F_n(s) \ge 0$ everywhere: it is zero when $D_s \cap P_n$ is either empty or all of $P_n$, and strictly positive otherwise. Consequently, $F_0 = \sum_{n=1}^{N-1} F_n$ is nonnegative on $S$, and hence on $S_0$. In particular, for $s \in [q_{n+1}, q_n]$, $F_n(s) = 0$ occurs precisely when $s = q_{n+1}$ or $s = q_n$.

Set
\[
F_0 := \sum_{n=1}^{N-1} F_n.
\]
Thus $F_0$ is modular, nonnegative, and vanishes at every element of the chosen chain.

\section{Local to Global on $S_0$}

\subsection{Lemma 3 (Vanishing on $Q$)}

\begin{lemma}\label{lem:vanish}
A modular function on a finite distributive lattice that vanishes on some maximal chain vanishes identically.
\end{lemma}

This classical fact is proved in Appendix~\ref{app:classical}. Applying Lemma~\ref{lem:vanish} to the restriction of $F_0$ to $Q$ gives $F_0 \equiv 0$ on all of $Q$.

\subsection{Lemma 4 (Positivity outside $Q$)}

\begin{lemma}\label{lem:positive}
We have $F_0(s) > 0$ for every $s \in S_0 \setminus Q$.
\end{lemma}

\begin{proof}
Let
\[
B := \{ s \in S_0 \setminus Q : F_0(s) = 0 \}.
\]
Suppose, to the contrary, that $B \ne \emptyset$. Among all intervals whose endpoints lie on the chosen maximal chain and that contain an element of $B$, let
\[
I^* = [q^*, q^{**}]
\]
be one of minimum chain length. Such an interval exists since $S_0 = [q^-, q^+]$ has both endpoints on the chosen chain.

Choose $s^* \in I^* \cap B$.

If $q^*$ and $q^{**}$ were consecutive, Lemma~\ref{lem:weight} and the construction of the corresponding local function would give $F_0(s^*) > 0$, contradicting $s^* \in B$. Hence they are not consecutive on the chosen chain, so there is a chain element
\[
q \in (q^*, q^{**}).
\]
Set
\[
u = s^* \wedge q, \qquad v = s^* \vee q.
\]
Then
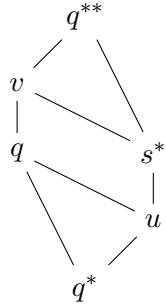
\[
u \in [q^*, q], \qquad v \in [q, q^{**}].
\]

\begin{center}
\begin{tikzpicture}[scale=0.9]
\node (qss) at (0,3) {$q^{**}$};
\node (v) at (-1,2) {$v$};
\node (s) at (1,1) {$s^*$};
\node (q) at (-1,1) {$q$};
\node (u) at (1,0) {$u$};
\node (qs) at (0,-1) {$q^*$};
\draw (qss)--(v)--(q)--(qs);
\draw (qss)--(s)--(u)--(qs);
\draw (v)--(s);
\draw (q)--(u);
\end{tikzpicture}
\end{center}
\noindent Schematic order diagram of the elements $q^*, u, s^*, q, v, q^{**}$ used in the proof of Lemma~\ref{lem:positive}. The pair $s^*, q$ decomposes the interval $[u,v]$, while $[q^*, q]$ and $[q, q^{**}]$ are intervals whose endpoints lie on the chosen maximal chain. Lines indicate order relations, not necessarily covers.

By modularity and $F_0(q) = F_0(s^*) = 0$,
\[
F_0(u) + F_0(v) = 0.
\]
Since $F_0$ is nonnegative on $S_0$, it follows that
\[
F_0(u) = F_0(v) = 0.
\]
If $u \notin Q$, then $u \in B$; but $u \in [q^*, q]$, an interval with chain endpoints and smaller chain length than $I^*$, contradicting the minimality of $I^*$. Hence $u \in Q$. Similarly, $v \in Q$.

Now $s^*$ and $q$ form a decomposition of the interval $[u,v]$. Since $u, v, q \in Q$ and $Q$ is decomposition-closed, it follows that $s^* \in Q$, contradicting $s^* \in B$. Therefore $B = \emptyset$.
\end{proof}

Together, Lemmas~\ref{lem:vanish} and~\ref{lem:positive} give $F_0^{-1}(0) = Q$ exactly, on $S_0$.

\section{Global Assembly}

To extend the construction from $S_0$ to all of $S$, define
\[
F_{\mathrm{top}}(s) := |D_s \setminus D_{q^+}|.
\]
Clearly $F_{\mathrm{top}}(s) \ge 0$, with equality iff $s \le q^+$.
\[
F_{\mathrm{bot}}(s) := |D_{q^-} \setminus D_s|.
\]
Likewise, $F_{\mathrm{bot}}(s) \ge 0$, with equality iff $s \ge q^-$.

Both are modular (weight-sums with $g \equiv 1$ on the relevant region) and nonnegative by construction. Together with $F_0$ from Sections~6--7, they give the global function
\[
F := F_{\mathrm{top}} + F_0 + F_{\mathrm{bot}}.
\]
$F$ is modular (sum of modular functions) and $F \ge 0$ on all of $S$. If $s \notin Q$: either $s \not\le q^+$ or $s \not\ge q^-$, in which case $F_{\mathrm{top}}(s) > 0$ or $F_{\mathrm{bot}}(s) > 0$ implies $F(s) > 0$; or $q^- \le s \le q^+$, i.e., $s \in S_0 \setminus Q$, in which case $F_0(s) > 0$ by Lemma~\ref{lem:positive}. Conversely if $s \in Q$, then $q^- \le s \le q^+$ so $F_{\mathrm{top}}(s) = F_{\mathrm{bot}}(s) = 0$, and $F_0(s) = 0$ by Lemma~\ref{lem:vanish}. Hence
\[
F^{-1}(0) = Q.
\]
Hence $Q = \arg\min F$, completing the proof of Theorem~\ref{thm:main}. \qed

\section{Closing Remarks}

We have characterized the minimizer sets of modular functions on finite distributive lattices by a purely lattice-theoretic closure property. The proof is constructive and yields an explicit modular function from any maximal chain of the given decomposition-closed sublattice.

Our characterization suggests that decomposition-closedness is an intrinsic lattice-theoretic notion, independent of its role in modular optimization. It would be interesting to investigate its role in other lattice-theoretic settings and to determine whether analogous characterizations hold for broader classes of lattices or for other classes of functions.

The result sits naturally alongside the classical fact that the minimizer set of a submodular function on a distributive lattice is always a sublattice, with every sublattice so achievable \cite{Fujishige2005,Topkis1998}, and the sharper recent result that every distributive lattice already arises from the more restrictive class of $M^\natural$-concave functions \cite{FujiiKijima2021}. Modular functions are strictly more restrictive still --- both sub- and supermodular --- and Theorem~\ref{thm:main} shows exactly how much this costs: the achievable minimizer sets shrink from all sublattices to the decomposition-closed ones.

A different, purely lattice-theoretic closure condition on sublattices --- closure under taking relative complements --- has recently been studied by Cz\'edli \cite{Czedli2022}. Decomposition-closedness is a distinct condition, triggered by a pair of elements already in the sublattice rather than by a fixed witness element; we are not aware of a previous study of decomposition-closedness in the present sense.

Finally, we note that the present characterization was developed in the course of our work on designing equitable and modular stable matching rules \cite{AlkanYildizWP}, where decomposition-closed sublattices of the (distributive) lattice of stable matchings describe exactly which subsets of stable outcomes can be selected by an additively separable criterion. We expect the characterization to be of independent use wherever a distributive lattice of outcomes is filtered by a linear or additive scoring rule.

\appendix

\section{Proof of Lemma~\ref{lem:vanish} (Classical Fact)}\label{app:classical}

\begin{proof}
Let $L$ be the lattice, $Z$ its join-irreducibles, $\omega : L \to \mathcal{P}(Z)$ the Birkhoff map, and $F$ modular on $L$ with weight $h : Z \to \mathbb{R}$, so
\[
F(t) = F(\max L) - \sum_{z \in Z \setminus \omega(t)} h(z).
\]
Let $t_1 > t_2 > \cdots > t_M$ be a maximal chain on which $F$ vanishes. Since the chain is maximal, each step is a cover, so $\omega(t_{n+1}) = \omega(t_n) \setminus \{z_n\}$ for a single $z_n \in Z$, and as $n$ ranges over $1, \dots, M-1$, the $z_n$ range over all of $Z$, each exactly once, by the well-known correspondence between maximal chains of a distributive lattice and linear extensions of its join-irreducibles. Then
\[
h(z_n) = F(t_n) - F(t_{n+1}) = 0 - 0 = 0
\]
for every $n$, so $h \equiv 0$ on $Z$. Hence $F(t) = F(\max L) - 0 = F(\max L)$ for every $t \in L$; since $F$ vanishes at $t_1 = \max L$, $F \equiv 0$ on $L$.
\end{proof}

\section{The $N_5$ Counterexample}\label{app:N5}

We exhibit a decomposition-closed sublattice of a finite (nondistributive) lattice that is not the minimizer set of any modular function, showing that distributivity in Theorem~\ref{thm:main} cannot be dropped.

Let $N_5 = \{0, x, y, z, 1\}$ be the pentagon lattice, with $0 \lessdot x \lessdot y \lessdot 1$ a three-step chain, $0 \lessdot z \lessdot 1$, and $x, z$ incomparable, $y, z$ incomparable. Figure~\ref{fig:N5} shows its Hasse diagram.

$N_5$ is the standard smallest nonmodular (hence nondistributive) lattice.

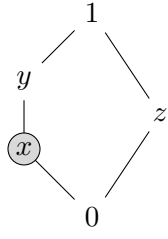
\begin{figure}[h]
\centering
\begin{tikzpicture}[scale=0.9]
\node (one) at (0,3) {$1$};
\node (y) at (-1,2) {$y$};
\node (z) at (1,1.5) {$z$};
\node (zero) at (0,0) {$0$};
\draw (one)--(y);
\draw (y) -- (-1,1);
\draw (-1,1) -- (zero);
\draw (one)--(z)--(zero);
\node[circle,draw,fill=gray!30,inner sep=1.5pt] (x) at (-1,1) {$x$};
\end{tikzpicture}
\caption{The pentagon lattice $N_5$. The shaded node marks $Q = \{x\}$, which Proposition~\ref{prop:N5} shows is decomposition-closed but not the minimizer set of any modular function.}
\label{fig:N5}
\end{figure}

\begin{proposition}\label{prop:N5}
$Q = \{x\}$ is a decomposition-closed sublattice of $N_5$ that is not the minimizer set of any modular function on $N_5$.
\end{proposition}

\begin{proof}
Since $Q$ is a singleton, it is trivially a sublattice; and the only decomposition of the interval $[x,x]$ is $s = s' = x$ (as $s \wedge s' = s \vee s' = x$ forces $x \le s, s' \le x$ and hence $s = s' = x$), which lies in $Q$. So $Q$ is decomposition closed.

Now let $F : N_5 \to \mathbb{R}$ be modular. Every pair of comparable elements gives the trivial identity $F(s) + F(s') = F(s) + F(s')$, so the modularity equation is only informative on the two incomparable pairs: $(x,z)$, with $x \wedge z = 0$, $x \vee z = 1$, and $(y,z)$, with $y \wedge z = 0$, $y \vee z = 1$. These give
\[
F(x) + F(z) = F(0) + F(1) = F(y) + F(z),
\]
so $F(x) = F(y)$ for every modular $F$ on $N_5$. Consequently $x$ can never be the unique minimizer of a modular function: whenever $x \in \arg\min F$, so is $y$. Hence no modular function has minimizer set exactly $\{x\}$.
\end{proof}

This obstruction illustrates where the distributive hypothesis enters: $N_5$ is not isomorphic to the lattice of downsets of its poset of join-irreducibles, so the weight-based Birkhoff representation used in Section~3 is unavailable.

\section{A Worked Example}\label{app:example}

We illustrate the construction of Theorem~\ref{thm:main} on a small concrete lattice.

Let $S$ be the divisors of $12$ ordered by divisibility, with $s \wedge s' = \gcd(s,s')$ and $s \vee s' = \operatorname{lcm}(s,s')$. Writing each divisor by its exponents of $2$ and of $3$, $S$ is isomorphic to the product of chains $\{0,1,2\} \times \{0,1\}$:
\[
1 = (0,0), \quad 2 = (1,0), \quad 4 = (2,0), \quad 3 = (0,1), \quad 6 = (1,1), \quad 12 = (2,1).
\]
Figure~\ref{fig:12} shows the Hasse diagram. $S$ is distributive, with join-irreducibles $R = \{a_1 \lessdot a_2\} \cup \{b_1\}$: $a_1, a_2$ correspond to reaching exponent $1$, respectively $2$, of the prime $2$ (i.e.\ to the elements $2$ and $4$), and $b_1$ corresponds to reaching exponent $1$ of the prime $3$ (i.e.\ to the element $3$).

\begin{figure}[h]
\centering
\begin{tikzpicture}[scale=0.9]
\node (s2) at (-1,1) {$2$};
\node (s6) at (1,2) {$6$};
\draw (0,3)--(-1,2)--(s2)--(0,0);
\draw (0,3)--(s6)--(1,1)--(0,0);
\draw (-1,2)--(s6);
\draw (s2)--(1,1);
\node[circle,draw,fill=gray!30,inner sep=1.5pt] (t12) at (0,3) {$12$};
\node[circle,draw,fill=gray!30,inner sep=1.5pt] (f4) at (-1,2) {$4$};
\node[circle,draw,fill=gray!30,inner sep=1.5pt] (s3) at (1,1) {$3$};
\node[circle,draw,fill=gray!30,inner sep=1.5pt] (s1) at (0,0) {$1$};
\end{tikzpicture}
\caption{The divisor lattice of $12$. Shaded nodes form $Q = \{1,3,4,12\}$. The modular function $F$ constructed below vanishes exactly on the shaded nodes and equals $1$ on the two unshaded ones, $2$ and $6$.}
\label{fig:12}
\end{figure}
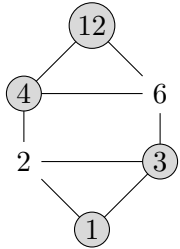

Let $Q = \{1,3,4,12\}$, shaded in Figure~\ref{fig:12}. Here $3$ and $4$ are incomparable, with $3 \wedge 4 = 1$ and $3 \vee 4 = 12$, so $Q$ is a sublattice with $q^- = 1$ and $q^+ = 12$; in fact $q^-$ and $q^+$ are the bottom and top of $S$ itself, so $S_0 = [q^-, q^+] = S$. Unlike the interval-shaped examples above, $Q$ is now a proper subset of $S_0$: the elements $2$ and $6$ lie in $S_0 \setminus Q$, and it is exactly on these two elements that Lemma~\ref{lem:positive} will have to do real work.

Moreover, $Q$ is decomposition closed: the only nontrivial decomposition of an interval with endpoints in $Q$ is the pair $3, 4$, which decomposes $[1,12]$, and both elements belong to $Q$.

Take the maximal chain $q_1 = 12 > q_2 = 4 > q_3 = 1$ of $Q$ (the element $3$ is off this chain, and is handled by Lemma~\ref{lem:vanish} below). The two $Q$-prime intervals are $I_1 = [4,12]$ and $I_2 = [1,4]$.

$I_1 = [4,12]$ is a single cover, with singleton difference poset $\{b_1\}$; as in the discussion above, this forces $F_1 \equiv 0$.

$I_2 = [1,4]$ is more interesting: it skips the element $2$, with difference poset $D_4 \setminus D_1 = \{a_1, a_2\}$, a connected two-element chain $a_1 \lessdot a_2$. Here Lemma~\ref{lem:weight} gives genuinely nonzero weights: orienting the spanning edge $a_1 \to a_2$, we get $g(a_1) = 1$, $g(a_2) = -1$. The resulting function $F_2(s) = \sum_{r \in D_s \cap \{a_1, a_2\}} g(r)$ evaluates to
\[
F_2(1) = 0, \quad F_2(2) = 1, \quad F_2(3) = 0, \quad F_2(4) = 0, \quad F_2(6) = 1, \quad F_2(12) = 0,
\]
since $D_2 \cap \{a_1,a_2\} = \{a_1\}$ and $D_6 \cap \{a_1,a_2\} = \{a_1\}$ each contribute $g(a_1) = 1$, while $D_4$ and $D_{12}$ contain all of $\{a_1,a_2\}$ (summing to $0$) and $D_1, D_3$ meet it not at all.

Since $q^- = 1$ and $q^+ = 12$ are the bottom and top of $S$, both boundary terms vanish, $F_{\mathrm{top}} \equiv F_{\mathrm{bot}} \equiv 0$, so $F = F_0 = F_1 + F_2 = F_2$:
\[
F(1) = 0, \quad F(2) = 1, \quad F(3) = 0, \quad F(4) = 0, \quad F(6) = 1, \quad F(12) = 0.
\]
As Lemma~\ref{lem:vanish} predicts, $F$ vanishes at the off-chain element $3$ too, even though $3$ played no role in the construction: modularity on the incomparable pair $(3,4)$ gives $F(3) + F(4) = F(1) + F(12) = 0$, and $F(4) = 0$ forces $F(3) = 0$. Meanwhile $F$ is strictly positive exactly on $S_0 \setminus Q = \{2,6\}$, as Lemma~\ref{lem:positive} guarantees.

One checks directly that $F$ is modular: the only incomparable pairs of $S$ are $(2,3), (4,3), (4,6)$, and
\[
F(2) + F(3) = 1 = F(1) + F(6),
\]
\[
F(4) + F(3) = 0 = F(1) + F(12),
\]
\[
F(4) + F(6) = 1 = F(2) + F(12).
\]
Thus $Q = \arg\min F$, exactly as Theorem~\ref{thm:main} predicts.


\begin{thebibliography}{9}

\bibitem{Birkhoff1937} G. Birkhoff, \emph{Rings of sets}, Duke Mathematical Journal \textbf{3} (1937), no.~3, 443--454. \url{https://doi.org/10.1215/S0012-7094-37-00334-X}

\bibitem{DaveyPriestley} B. A. Davey and H. A. Priestley, \emph{Introduction to Lattices and Order}, 2nd ed., Cambridge University Press, Cambridge, 2002. (See Chapter 5, ``Representation: The Finite Case,'' pp.~112--124.)

\bibitem{Fujishige2005} S. Fujishige, \emph{Submodular Functions and Optimization}, 2nd ed., Annals of Discrete Mathematics 58, Elsevier, Amsterdam, 2005.

\bibitem{Topkis1998} D. M. Topkis, \emph{Supermodularity and Complementarity}, Princeton University Press, Princeton, NJ, 1998.

\bibitem{FujiiKijima2021} T. Fujii and S. Kijima, \emph{Representing the minimizer sets of distributive lattices by $M^\natural$-concave functions}, Operations Research Letters \textbf{49} (2021), no.~5, 664--668.

\bibitem{Czedli2022} G. Cz\'edli, \emph{A property of lattices of sublattices closed under taking relative complements and its connection to 2-distributivity}, Mathematica Pannonica \textbf{28} (2022), 109--117. DOI: 10.1556/314.2022.00014.

\bibitem{AlkanYildizWP} A. Alkan and K. Yildiz, \emph{Equitable Stable Matchings Under Modular Assessment}, working paper.

\end{thebibliography}
\end{document}